\documentclass[11pt]{article}

\usepackage{amsmath, amssymb, amsthm}
\usepackage[all]{xy}
\usepackage[colorlinks]{hyperref}
\usepackage{enumerate}
\usepackage[nohug]{diagrams}
\usepackage{MnSymbol, blkarray}

\newcommand{\Sptn}{\on{Sp}(2n,\mathbb{R})}
\newcommand{\Om}{\on{O}(m)}
\newcommand{\MtnmR}{\on{M}_{2n\times m}(\mathbb{R})}
\newcommand{\Un}{\on{U}(n)}
\newcommand{\Um}{\on{U}(m)}
\newcommand{\MnmC}{\on{M}_{n\times m}(\mathbb{C})}
\newcommand{\MnmR}{\on{M}_{n\times m}(\mathbb{R})}
\newcommand{\MnmRm}{\on{M}_{n\times m}^{\textrm{rk }m}(\mathbb{R})}
\newcommand{\MtnmRm}{\on{M}_{2n\times m}^{\textrm{rk }m}(\mathbb{R})}
\newcommand{\MnmCm}{\on{M}_{n\times m}^{\textrm{rk }m}(\mathbb{C})}

\newcommand{\sptn}{\mathfrak{sp}(2n,\mathbb{R})}
\newcommand{\om}{\mathfrak{o}(m)}
\newcommand{\un}{\mathfrak{u}(n)}
\newcommand{\um}{\mathfrak{u}(m)}
\newcommand{\mfg}{\mathfrak{g}}
\newcommand\gl{\mathfrak {gl}}
\newcommand\GL{\on{GL}}
\renewcommand\sp{\mathfrak {sp}}

\newcommand\on{\operatorname}
\newcommand\Ad{\on{Ad}}
\newcommand\ad{\on{ad}}
\newcommand\rank{\on{rank}}
\newcommand\Emb{\on{Emb}}
\newcommand\ham{\on{ham}}
\newcommand\Diff{\on{Diff}}
\newcommand\Tr{\on{Tr}}
\newcommand\vol{\on{vol}}
\newcommand\RR{\mathbb{R}}
\newcommand\CC{\mathbb{C}}
\newcommand\J{\mathrm{J}}
\newcommand\X{\mathfrak X}

\newtheorem{theorem}{Theorem}[section]
\newtheorem{corollary}[theorem]{Corollary}
\newtheorem{lemma}[theorem]{Lemma}
\newtheorem{proposition}[theorem]{Proposition}
\theoremstyle{definition}
\newtheorem{definition}[theorem]{Definition}
\newtheorem{remark}[theorem]{Remark}
\newtheorem{example}[theorem]{Example}

\begin{document}

\title{Dual pairs for matrix groups}
\author{Paul Skerritt$^{1}$ and Cornelia Vizman$^{2}$ \\
\\
\emph{In honour of Darryl Holm's 70\textsuperscript{th} birthday.}}

\addtocounter{footnote}{1}
\footnotetext{Department of Mathematics, University of Surrey, Guildford GU2 7XH, United Kingdom.
\texttt{p.skerritt@surrey.ac.uk}
\addtocounter{footnote}{1}}

\footnotetext{Department of Mathematics, West University of Timi\c soara. RO--300223 Timi\c soara. Romania.
\texttt{cornelia.vizman@e-uvt.ro}
\addtocounter{footnote}{1} }

\date{ }
\maketitle

\vspace{-.7cm}

\begin{abstract}
In this paper we present two dual pairs  that can be seen as the linear analogues of the following two dual pairs 
related to fluids: the EPDiff dual pair due to Holm and Marsden,
and the ideal fluid dual pair due to Marsden and Weinstein.
%Matrix analogues of the ideal fluid dual pair and EPDiff dual pair
\end{abstract}

%\tableofcontents

%%%%%%%%%%%%%%%%%%%%%%%
%%%%%%%%%%%%%%%%

\section{Introduction}

\subsection{Definitions and new results}

Let $M$ be a symplectic manifold, $P_1, P_2$ Poisson manifolds, and let $\on{J}_1, \on{J}_2$ be a pair of Poisson surjective submersions
\[
P_1\stackrel{\on{J}_1}{\leftarrow} M \stackrel{\on{J}_2}{\rightarrow} P_2,
\]
The maps $\on{J}_1, \on{J}_2$ are said to form a \emph{Lie-Weinstein dual pair} \cite{Weinstein1983, OrtegaRatiu2004} if
the tangent distributions to the fibres of $\on{J}_1, \on{J}_2$ are symplectically orthogonal, i.e., 
\[
(\ker T\on{J}_1)^\omega = \ker T\on{J}_2.
\]
They are said to form a \emph{Howe dual pair} \cite{OrtegaRatiu2004} if the Poisson subalgebras $\on{J}_1^*(C^\infty(P_1))$ and $\on{J}_2^*(C^\infty(P_2))$ of $(C^\infty(M), \lbrace\cdot,\cdot\rbrace)$ centralise one another. Under mild conditions \cite[Proposition 11.1.3]{OrtegaRatiu2004}, the definitions of Lie-Weinstein and Howe dual pairs may be shown to be equivalent.
When such dual pairs exist, the Poisson structures in $P_1$ and $P_2$ are closely related. In particular, it can be shown (see for example \cite[Theorem E.13]{Blaom2001}) that there is a one-to-one correspondence between symplectic leaves of $P_1$ and $P_2$.

Several authors have considered the analogue of dual pairs in situations where $\on{J}_1, \on{J}_2$ are not necessarily submersions, and have suggested conditions under which the above-mentioned one-to-one correspondence between symplectic leaves in $P_1$ and $P_2$ still holds. We mention in particular the work of Ortega  \cite{Ortega2003, OrtegaRatiu2004} on singular dual pairs, and that of Balleier and Wurzbacher \cite{BalleierWurzbacher2012} on the centralising Howe condition without the submersion property.
In this paper, we will use the term ``dual pair'' to mean in this generalised sense, satisfying many of the desired properties of dual pairs, but not necessarily Lie-Weinstein. Specifically, we are concerned with the situation where a dual pair on the symplectic manifold $M$ arises from a pair of Hamiltonian actions of finite-dimensional Lie groups $G_1, G_2$, with corresponding equivariant momentum maps $\on{J}_1:M\rightarrow \mathfrak{g}_1^*$ and $\on{J}_2:M\rightarrow \mathfrak{g}_2^*$. 

We first discuss a criterion we call \emph{mutual transitivity}, meaning the fibres of $\on{J}_1$ are $G_2$-orbits and the fibres of $\on{J}_2$ are $G_1$-orbits. We show that when this criterion is satisfied, 
the coadjoint orbits in $\on{J}_1(M)\subset \mathfrak{g}_1^*$ and $\on{J}_2(M)\subset \mathfrak{g}_2^*$ have a one-to-one correspondence. Additionally, any reduced space for the $G_1$-action is symplectomorphic to a coadjoint orbit in $\on{J}_2(M)$, and similarly with 1 and 2 switched.

We then give several examples of mutually transitive dual pairs. We initially discuss a dual pair first considered by Balleier and Wurzbacher \cite{BalleierWurzbacher2012}. We then construct two Lie-Weinstein dual pairs, inspired by dual pairs related to fluid mechanics, and describe explicitly the (co)adjoint orbit correspondence between their images. The first of these was considered in \cite[pp.502-506]{KazhdanKostantSternberg1978}, although the full adjoint orbit correspondence was not provided there. The second is to our knowledge novel. We also point out some interesting relations between the momentum maps of the three examples considered, reminiscent of the seesaw pairs of dual pairs \cite{Kudla1984}.

For other discussions of the relationship between reduced spaces and coadjoint orbits in matrix group dual pairs, see \cite{BalleierWurzbacher2012, KazhdanKostantSternberg1978, LermanMontgomerySjamaar1993}. In \cite[Proposition 2.6]{BalleierWurzbacher2012}, the authors show that mutual transitivity is a consequence of the \emph{symplectic Howe condition} (stating $\on{J}_1^*(C^\infty(\mfg_1^*))$ and $\on{J}_1^*(C^\infty(\mfg_2^*))$ centralise one another), plus properness of both actions. Since in both of our main examples the action is not proper (see Sections \ref{sec:mtSpO} and \ref{sec:mtGLnGLm}), we focus instead on mutual transitivity.

\subsection{Motivation for Weinstein's definition of dual pairs}

Inspired by the work of Lie \cite{Lie1890}, Kazdhan, Kostant, and Sternberg \cite{KazhdanKostantSternberg1978}, and Howe \cite{Howe1989}, dual pairs were introduced by Weinstein \cite{Weinstein1983} in the context of the following problem. Suppose $P$ is a Poisson manifold, and we wish to find a canonical form for the Hamiltonian flow generated by some $h\in C^\infty(P)$. One approach to doing so is to introduce a so-called \emph{symplectic realisation} of $P$, which is a surjective Poisson map $\on{J}: M\rightarrow P$, where $M$ is some symplectic manifold. Since $\on{J}$ is Poisson, it respects dynamics in the sense that it intertwines the Hamiltonian flow in $M$ generated by $h\circ \on{J}$ and the Hamiltonian flow in $P$ generated by $h$. Since Hamiltonian flows in a symplectic manifold always have a local canonical form, expressed in Darboux coordinates, this gives a local canonical form for the Hamiltonian flow in $P$. This is the essential idea behind the introduction of \emph{Clebsch variables} in continuum problems \cite{MarsdenWeinstein1983, HolmMarsden2004}.

Suppose further that $\on{J}:M\rightarrow P$ is a submersion with connected fibres, and let $\mathcal{D}$ denote the (regular) foliation defined by these fibres. Let $D:=T\mathcal{D}$ be the corresponding involutive distribution, and $D^\omega$ the symplectically orthogonal distribution. Using the submersion and Poisson properties of $\on{J}$, it may be shown that $D^\omega$ is also involutive, and so integrates to a regular foliation $\mathcal{D}^\omega$. If we further suppose $M/\mathcal{D}^\omega$ has a smooth structure making the projection $\pi:M\rightarrow M/ \mathcal{D}^\omega$ a submersion, then $M/\mathcal{D}^\omega$ may be given a Poisson structure with respect to which $\pi$ is Poisson, and we end up with a with a pair of Poisson maps $P\stackrel{\on{J}}{\leftarrow} M \stackrel{\pi}{\rightarrow} M/\mathcal{D}^\omega$ with symplectically orthogonal fibres. Then the Hamiltonian flow in $M$ generated by $h\circ\on{J}$ preserves the leaves of $\mathcal{D}^\omega$, and so trajectories of this flow are mapped by $\pi$ to points of $M/\mathcal{D}^\omega$. Hence in some sense $P$ describes the dynamics of the problem, while $M/\mathcal{D}^\omega$ describes the dynamical invariants of its symplectic realisation.

The above construction motivates the general definition of a Lie-Weinstein dual pair 
$P_1\stackrel{\on{J}_1}{\leftarrow} M \stackrel{\on{J}_2}{\rightarrow} P_2$.
In general the fibres of $\on{J}_1, \on{J}_2$ are orbits (at least formally) of infinite-dimensional Lie groups, with Lie algebras isomorphic to the Poisson subalgebras $\on{J}_2^*(C^\infty(P_2))$ and $\on{J}_1^*(C^\infty(P_1))$ of the Poisson algebra $(C^\infty(M), \lbrace\cdot,\cdot \rbrace)$. In this paper we are interested in the case where the dual pair arises from a pair of Hamiltonian Lie group actions of groups $G_1, G_2$, with corresponding equivariant momentum maps $\on{J}_i:M\rightarrow \on{J}_i(M)\subset \mathfrak{g}_i^*$.

\subsection{Motivation for the particular dual pairs considered in this paper}

The original motivation for this paper was to find two dual pairs that can be seen as 
the linear analogues of the following two dual pairs 
related to fluids: the EPDiff dual pair introduced by Holm and Marsden in
\cite{HolmMarsden2004} and the ideal fluid dual pair introduced by Marsden and Weinstein in \cite{MarsdenWeinstein1983} (proven to be indeed dual pairs in \cite{GayBalmazVizman2011}, where the additional technicalities in defining infinite-dimensional dual pairs are discussed). EPDiff stands here for Euler-Poincar\'e equation on the 
diffeomorphism group.

The EPDiff dual pair involves the manifold of embeddings $\Emb(S,N)$ of a compact manifold $S$ into a manifold $N$, acted on by the diffeomorphism groups $\Diff(N)$ 
from the left and by $\Diff(S)$ from the right. The momentum maps for the lifted cotangent actions,
 restricted to the open subset $T^*\Emb(S,N)^\times$ of $T^*\Emb(S,N)$
 that consists of nowhere zero 1-form densities, define a dual pair\footnote{See \cite{GayBalmazVizman2011} for the precise definition of dual pair in the infinite-dimensional context.}
\begin{equation*}
\X(N)^*\stackrel{\mathrm{J}_{L}}{\longleftarrow} T^*\Emb(S,N)^\times
\stackrel{\mathrm{J}_{R}}{\longrightarrow} \X(S)^*.
\end{equation*}
%where $\X(N)^*_{\on{J}_L}$ denotes the image of $\on{J}_L$ in the dual Lie algebra $\X(N)^*$, and similarly for $\X(S)^*_{\on{J}_R}$.

The linear analogues of these actions are the left $\GL(n,\RR)$-action and the right
$\GL(m,\RR)$-action on the manifold  of rank $m$ matrices $\MnmRm$
(identified with linear injective maps $:\RR^m\to\RR^n$).
We show in Section \ref{s5} that the lifted cotangent momentum maps,
restricted to an open subset
% $\MnmRm\times \MnmRm$ 
of $T^*\MnmRm$, define a dual pair
\begin{equation*}
\gl(n,\RR)^*_{\on{J}_L}\stackrel{\mathrm{J}_{L}}{\longleftarrow} \MnmRm\times \MnmRm
\stackrel{\mathrm{J}_{R}}{\longrightarrow} \gl(m,\RR)^*_{\on{J}_R},
\end{equation*}
where $\gl(n,\RR)^*_{\on{J}_L}$ denotes the image of $\on{J}_L$ in the dual Lie algebra $\gl(n,\RR)^*$, and similarly for $\gl(m,\RR)^*_{\on{J}_R}$.

For the ideal fluid dual pair one notices that, given a compact manifold $S$
endowed with a volume form $\mu$, and a manifold $M$ endowed with an exact symplectic form $\omega$, 
the group $\Diff_{\vol}(S)$ of volume preserving diffeomorphisms and the group $\Diff_{\ham}(M)$ of Hamiltonian diffeomorphisms  act in a Hamiltonian way on
$\Emb(S,M)$. The symplectic form considered on $\Emb(S,M)$
is built in a natural way with the differential forms $\mu$ and $\omega$.
The momentum map for the right $\Diff_{\vol}(S)$-action together with the momentum map for the left action of the quantomorphism group,
a one dimensional central extension of $\Diff_{\ham}(M)$ that integrates $C^\infty(M)$, define a dual pair
\begin{equation*}
C^\infty(M)^*\stackrel{\mathrm{J}_{L}}{\longleftarrow} \Emb(S,M)^\times
\stackrel{\mathrm{J}_{R}}{\longrightarrow} \X_{\vol}(S)^*.
\end{equation*}
%Here the quantomorphism group is the central extension of $ \Diff_{\ham}(M)$ by the group $2$-cocycle 
%$B(\ph_1,\ph_2):=\int_{x}^{\ph_2(x)}\left(\theta-\ph_1^*\theta\right)$ for some $x\in M$.

In the special case $M=T^*N$ the inclusion of $\Emb(S,M)$ in $T^*\Emb(S,N)$
naturally induced by the volume form $\mu$ is symplectic.
Thus the linear analogue of the symplectic manifold $\Emb(S,M)$
is the manifold of rank $m$ matrices
$\MtnmRm$
(identified with linear injective maps $:\RR^m\to\RR^{2n}$) and
endowed with symplectic form induced by the cotangent symplectic form
on $T^*\MnmR$, namely the linear symplectic form on the vector space $\MtnmR$ 
\[
\Omega(E,F) := \Tr(E^\top\mathbb{J}F),\quad\text{ where } \mathbb{J} = \begin{bmatrix} 0_n & I_n \\ -I_n & 0_n\end{bmatrix}.
\]
It is not difficult to see that the maximal Lie subgroups of $\GL(2n,\RR)$ and $\GL(m,\RR)$ that preserve this symplectic form are the real symplectic group $\Sptn$ and the orthogonal group $\Om$. We show in Section \ref{s4} that the momentum maps for these two actions define a dual pair 
\begin{equation*}
\sptn^*_{\on{J}_L}\stackrel{\mathrm{J}_{L}}{\longleftarrow} \MtnmRm
\stackrel{\mathrm{J}_{R}}{\longrightarrow} \om^*_{\on{J}_R}.
\end{equation*}
In the special case $m=2n$ this dual pair appears in \cite{Skerritt} (in the context of semiclassical quantum mechanics).

\subsection{Outline of paper}
In Section \ref{s2}, we define the notion of mutually transitive actions, describe the coadjoint orbit and coadjoint orbit-reduced space correspondences, and discuss the relation between mutually transitivity and Lie-Weinstein dual pairs. In Section \ref{s3}, we describe the $(\Un,\Um)$ dual pair, first considered by Balleier and Wurzbacher \cite{BalleierWurzbacher2012}, and prove it satisfies mutual transitivity. In Section \ref{s4}, we construct the $(\Sptn, \Om)$ dual pair, which is the analogue of the ideal fluid dual pair, prove it satisfies mutual transitivity, explicitly describe the (co)adjoint orbit correspondence, and point out some connections with the $(\Un,\Um)$ dual pair. In Section \ref{s5}, we construct the $(\GL(n,\RR), \GL(m,\RR))$ dual pair, which is the analogue of the EPDiff dual pair, prove it satisfies mutual transitivity, explicitly describe the (co)adjoint orbit correspondence, and point out some connections with the $(\Un,\Um)$ dual pair.

%%%%%%%%%%%%%%%%

\section{Mutually transitive actions and dual pairs}\label{s2}

In this section, we first introduce the notion of \emph{mutually transitive actions}, and indicate the resulting coadjoint orbit correspondence. We then describe how mutual transitivity allows us to view reduced spaces of one action as coadjoint orbits of the other. We emphasise here that by contrast with other treatments in the literature, this correspondence invokes only smoothness of the actions, and does not require properness. Finally we outline the relationship between mutual transitivity and Lie-Weinstein dual pairs.

Propositions \ref{prop:coadinv}, Lemma \ref{lemma:mommappullback}, and Proposition \ref{prop:redcoad} are also proved in \cite[Theorem 2.9]{BalleierWurzbacher2012}. We choose to include them here for completeness, since the first two are short, while our treatment of the last differs somewhat from that in \cite{BalleierWurzbacher2012}.

A fuller treatment of dual pairs and related concepts can be found in \cite[Section IV.7]{LibermannMarle1987}, \cite[Chapter 11]{OrtegaRatiu2004}, and \cite{BalleierWurzbacher2012}.

\subsection{Mutually transitive actions}

Let $(M,\omega)$ be a symplectic manifold, and let $\Phi_1:G_1\times M \rightarrow M$ and $\Phi_2:G_2\times M\rightarrow M$ be symplectic actions. We assume $M$, $G_1$, and $G_2$ are all finite-dimensional (dual pairs in infinite dimensions are discussed in \cite{GayBalmazVizman2011}).

\begin{definition}
We say the actions $\Phi_1, \Phi_2$ are \emph{mutually transitive} if the following three properties hold:
\begin{itemize}
\item
$\Phi_1$ and $\Phi_2$ commute,
\item
$\Phi_1$ and $\Phi_2$ are Hamiltonian actions, with corresponding equivariant momentum maps $\on{J}_1:M\to\mfg_1^*$ and $\on{J}_2:M\to\mfg_2^*$,
\item
each level set of $\on{J}_1$ is a $G_2$-orbit and vice versa, i.e., for any $x\in M$,
\[
\on{J}_1^{-1}(\on{J}_1(x)) = G_2\cdot x \qquad\textrm{and}\qquad \on{J}_2^{-1}(\on{J}_2(x)) = G_1\cdot x.
\]
\end{itemize}
\end{definition}

Denoting the coadjoint orbit in $\mfg_i^*$ through $\mu_i$ by $\mathcal{O}_{\mu_i}$, we then have

\begin{proposition}\label{prop:coadinv}
Let $\Phi_1, \Phi_2$ be mutually transitive actions, with equivariant momentum maps $\on{J}_1, \on{J}_2$. Then for all $x\in M$,
\[
\on{J}_1^{-1}(\mathcal{O}_{\on{J}_1(x)}) = \on{J}_2^{-1}(\mathcal{O}_{\on{J}_2(x)}).
\]
\end{proposition}

\begin{proof}
\begin{align*}
\on{J}_1^{-1}(\mathcal{O}_{\on{J}_1(x)}) &= G_1\cdot \on{J}_1^{-1}(\on{J}_1(x)) \qquad\textrm{since $\on{J}_1$ is $G_1$-equivariant} \\
&= G_1\cdot (G_2\cdot x) \qquad\textrm{since $\Phi_2$ is transitive on the fibres of $\on{J}_1$}\\
&= G_2\cdot (G_1\cdot x) \qquad \textrm{since the actions $\Phi_1$ and $\Phi_2$ commute}\\
&= G_2\cdot\on{J}_2^{-1}(\on{J}_2(x)) \qquad\textrm{since $\Phi_1$ is transitive on the fibres of $\on{J}_2$}\\
&= \on{J}_2^{-1}(\mathcal{O}_{\on{J}_2(x)}) \qquad \textrm{since $\on{J}_2$ is $G_2$-equivariant}.
\end{align*}
\end{proof}

\begin{corollary}[{\cite[Theorem 2.9(i)]{BalleierWurzbacher2012}}]\label{coro}
In this situation, there exists a one-to-one correspondence between coadjoint orbits in $\on{J}_1(M)$ and $\on{J}_2(M)$ given by
\[
\mathcal{O}_{\mu_1}\mapsto \on{J}_2(\on{J}_1^{-1}(\mathcal{O}_{\mu_1})) { =\on{J}_2(\on{J}_1^{-1}(\mu_1))}
\]
or equivalently
\[
\qquad\qquad\mathcal{O}_{\on{J}_1(x)}\mapsto \mathcal{O}_{\on{J}_2(x)} \qquad\textrm{for }x\in M.
\]
\end{corollary}

\begin{proof}
Proposition \ref{prop:coadinv} shows that the defined map between coadjoint orbits is invertible. 
\end{proof}

\subsection{The relation between coadjoint orbits and reduced spaces}

Let $\Phi$ be any Hamiltonian action of $G$ on $M$ with equivariant momentum map $\on{J}$, and  let $\pi:M\rightarrow M/G$ denote the quotient map. For $\mu\in\on{J}(M)$, let $G_\mu$ denote the coadjoint stabiliser subgroup of $G$ at $\mu$, let $M_\mu$ be the set $\on{J}^{-1}(\mu)/G_\mu\simeq \on{J}^{-1}(\mathcal{O}_\mu)/G$, and let $\pi^\mu:\on{J}^{-1}(\mu)\rightarrow M_\mu\subset M/G$ denote the restriction of $\pi$ to $\on{J}^{-1}(\mu)$. In favourable situations (for example, if the group action $\Phi$ is free and proper), $M_\mu$ can be given a differentiable structure with respect to which $\pi^\mu$ is a submersion, and a symplectic structure $\omega_{M_\mu}$ satisfying $(\pi^\mu)^*\omega_{M_\mu} = (i^\mu)^*\omega$, where $i^\mu:\on{J}^{-1}(\mu)\hookrightarrow M$ is the inclusion.  The resulting symplectic manifold $(M_\mu, \omega_{M_\mu})$ is called the \emph{reduced space} or the \emph{Marsden-Weinstein-Meyer quotient} at $\mu\in\mfg^*$.

In this subsection, we demonstrate that as a consequence of mutual transitivity, such a differentiable and symplectic structure can always be defined on the reduced spaces corresponding to each action. Moreover, the reduced space $M_{\mu_1}$ of the $G_1$-action is symplectomorphic to a coadjoint orbit $\mathcal{O}_{\mu_2} \subset\mfg_2^*$ (for some related $\mu_2$). 

We first recall the concept of an \emph{initial submanifold}, and its relationship to group orbits.

\begin{definition}\cite[Section 1.1.8]{OrtegaRatiu2004} \label{defn:initialsub}
Let $M$ be a manifold, and $N$ a subset of $M$ endowed with its own manifold structure, such that the inclusion $i:N \hookrightarrow M$ is an immersion (i.e., $N$ is an immersed submanifold of $M$). We say $N$ is an \emph{initial submanifold} of $M$ if for any manifold $P$, a map $g:P\rightarrow N$ is smooth iff $i\circ g:P\rightarrow M$ is smooth.
\end{definition}

\begin{remark}
Stated differently, Definition \ref{defn:initialsub} says that if $h:P\rightarrow M$ is a smooth map with image contained in $N$, then $h$ corestricts to a smooth map $h':P\rightarrow N$.
\end{remark}

\begin{lemma}\rm{\cite[Proposition 2.3.12 (i)]{OrtegaRatiu2004}} \label{lemma:initialsub}
If $\Phi:G\times M\rightarrow M$ is a smooth $G$-action, then the orbit $G\cdot x$ through $x\in M$ is an initial submanifold of $M$.
\end{lemma}

Before proving our main result Proposition \ref{prop:redcoad}, we first give a supplementary lemma.

\begin{lemma}\label{lemma:mommappullback}
Let $\Phi_1, \Phi_2$ be mutually transitive actions, with equivariant momentum maps $\on{J}_1, \on{J}_2$.
Choose $x\in M$, and let $\mu_i=\on{J}_i(x),\ i=1,2$. Then the smooth map $\on{J}_2:M\rightarrow\mathfrak{g}_2^*$ restricts to a surjective submersion $\on{J}_2^{\mu_1}:\on{J}_1^{-1}(\mu_1)\rightarrow\mathcal{O}_{\mu_2}$ satisfying
\begin{equation}\label{eqn:mommappullback}
(i^{\mu_1})^*\omega = (\on{J}_2^{\mu_1})^*\omega^+_{\mathcal{O}_{\mu_2}},
\end{equation}
where $\omega^+_{\mathcal{O}_{\mu_2}}$ is the positive Kostant-Kirillov-Souriau form on the coadjoint orbit $\mathcal{O}_{\mu_2}\subset \mathfrak{g}_2^*$.
\end{lemma}

\begin{proof}
Since $\on{J}_1^{-1}(\mu_1) = G_2\cdot x$ is a $G_2$-orbit, and $\on{J}_2$ is $G_2$-equivariant, the image $\on{J}_2(\on{J}_1^{-1}(\mu_1))$ equals the coadjoint orbit $\mathcal{O}_{\mu_2}$. Since $\on{J}_1^{-1}(\mu_1)$ is a submanifold and $\mathcal{O}_{\mu_2}$ is an initial submanifold (by Lemma \ref{lemma:initialsub}), the restriction $\on{J}_2^{\mu_1}:\on{J}_1^{-1}(\mu_1)\rightarrow \mathcal{O}_{\mu_2}$ is smooth. By equivariance of $\on{J}_2$, it follows that $\on{J}_2^{\mu_1}$ is a submersion.

Now taking $\xi, \zeta \in\mathfrak{g}_2$, and using the equivariance of $\on{J}_2$, we have
\begin{align*}
\omega_x(\xi\cdot x, \zeta\cdot x) &= d_x\langle \on{J}_2,\xi\rangle (\zeta\cdot x) 
= \zeta\cdot x\, \langle\on{J}_2,\xi\rangle = \langle -\ad^*_{\zeta}\on{J}_2(x),\xi\rangle \\
&= \langle \on{J}_2(x), [\xi,\zeta]\rangle = (\omega_{\mathcal{O}_{\mu_2}}^+)_{\on{J}_2(x)}(-\ad_\xi^*\on{J}_2(x), -\ad_\zeta^*\on{J}_2(x)).
\end{align*}
%\[
%\omega_x(\xi\cdot x, \zeta\cdot x) = (\omega_{\mathcal{O}_{\mu_2}}^+)_{\mu_2}(-\ad_\xi^*\mu_2, -\ad_\zeta^*\mu_2), \qquad \xi,\zeta\in\mfg_2.
%\]
Again using equivariance of $\on{J}_2$, plus the fact that $\on{J}_1^{-1}(\mu_1)$ is a $G_2$-orbit, gives \eqref{eqn:mommappullback}.
\end{proof}

\begin{proposition}\label{prop:redcoad}
Let $\Phi_1, \Phi_2$ be mutually transitive actions, with equivariant momentum maps $\on{J}_1, \on{J}_2$.
Then any reduced space under the $G_1$-action is symplectomorphic to a coadjoint orbit in $\mathrm{J}_2(M)\subset\mathfrak{g}_2^*$, and similarly with 1 and 2 switched. Explicitly, 
for $x\in M$,
\[
M_{\mathrm{J}_1(x)} \simeq \mathcal{O}_{\mathrm{J}_2(x)}, \qquad M_{\mathrm{J}_2(x)}\simeq \mathcal{O}_{\mathrm{J}_1(x)},
\]
via a resp. $G_2$- and $G_1$-equivariant symplectomorphism.
\end{proposition}

\begin{proof} 
Again choose $x\in M$, and let $\mu_i=\on{J}_i(x),\ i=1,2$. By equivariance of $\on{J}_1$ and $\on{J}_2$ and the mutual transitivity property, it is not difficult to show that for any $y\in\on{J}_1^{-1}(\mu_1)$,
\[
(G_1)_{\on{J}_1(y)}\cdot y = (G_2)_{\on{J}_2(y)}\cdot y = G_1\cdot y\cap G_2\cdot y.
\]
Hence the fibres of the restrictions $\pi_1^{\mu_1}:\on{J}_1^{-1}(\mu_1)\rightarrow M_{\mu_1}$ and $\on{J}_2^{\mu_1}:\on{J}_1^{-1}(\mu_1)\rightarrow \mathcal{O}_{\mu_2}$ agree, and we get a bijection $\chi:M_{\mu_1}\rightarrow \mathcal{O}_{\mu_2}$ making the following diagram commute.
\begin{equation}\label{diag:redmom}
\begin{diagram}
&&\on{J}_1^{-1}(\mu_1)\\
&\ldTo^{\pi_1^{\mu_1}} &&\rdTo^{\on{J}_2^{\mu_1}}\\
M_{\mu_1} && \rTo^\chi && \mathcal{O}_{\mu_2}
\end{diagram}
\end{equation}

Pulling the smooth structure on $\mathcal{O}_{\mu_2}$ back to $M_{\mu_1}$ via $\chi$ implies that $\pi_1^{\mu_1}$ is also a smooth submersion. The fibres $(G_1)_{\mu_1}\cdot y$ of $\pi_1^{\mu_1}$ are integral manifolds of the degeneracy directions $(\mathfrak{g}_1)_{\mu_1}\cdot y$ of the restriction $(i^{\mu_1})^*\omega$. Then the usual Marsden-Weinstein-Meyer construction implies the existence of a reduced symplectic structure $\omega_{M_{\mu_1}}$ on  $M_{\mu_1}$ satisfying 
\begin{equation}\label{omega1}
(i^{\mu_1})^*\omega = (\pi_1^{\mu_1})^*\omega_{M_{\mu_1}}.
\end{equation}

By \eqref{eqn:mommappullback} and commutativity of diagram \eqref{diag:redmom}, we have
\begin{equation}\label{omega2}
(i^{\mu_1})^*\omega =  (\pi_1^{\mu_1})^*(\chi)^* \omega_{\mathcal{O}_{\mu_2}}^+.
\end{equation}
Then combining \eqref{omega1} and \eqref{omega2},
\[
(\pi_1^{\mu_1})^*\omega_{M_{\mu_1}} =(\pi_1^{\mu_1})^*(\chi)^*\omega_{\mathcal{O}_{\mu_2}}^+,
\]
and since $\pi_1^{\mu_1}$ is a surjective submersion,
\[
\omega_{M_{\mu_1}} = (\chi)^*\omega_{\mathcal{O}_{\mu_2}}^+,
\]
i.e., $\chi$ is a symplectomorphism. Since the $G_1$- and $G_2$-actions on $M$ commute, the $G_2$-action drops to $M_{\mu_1}$. Then commutativity of the diagram \eqref{diag:redmom} and $G_2$-equivariance of $\on{J}_2^{\mu_1}$ implies that $\chi$ is $G_2$-equivariant.

A similar argument shows that $M_{\mu_2} \subset M/G_2$ is symplectomorphic to $\mathcal{O}_{\mu_1}\subset \mfg_1^*$.
\end{proof}

The map $\chi$ in the above proof has a natural interpretation: using the identity $(i^{\mu_1})^*\omega = (\pi_1^{\mu_1})^*\omega_{M_{\mu_1}}$ it is easily shown that $\chi$ is the momentum map of the 
induced $G_2$-action on $M_{\mu_1}$.

\subsection{The relation to Lie-Weinstein dual pairs}

In this subsection we make contact with the notion of dual pair in the Weinstein's original sense \cite{Weinstein1983}. 

\begin{definition}\cite[Definition 11.1.1]{OrtegaRatiu2004}
Let $M$ be a symplectic manifold, and $P_1, P_2$ Poisson manifolds. A pair of Poisson maps
\begin{equation*}
P_1\xleftarrow{\J_1}M\xrightarrow{\J_2}P_2,
\end{equation*}
is called a \emph{Lie-Weinstein dual pair} if 
$\J_1, \J_2$ are surjective submersions satisfying 
\[
(\ker T\J_1)^\omega =  \ker T\J_2.
\]
%If in addition $\J_1, \J_2$ are surjective submersions onto their images, we say that the dual pair is \emph{full}.
\end{definition}

\begin{proposition}\label{prop:dualpairrank}
Let $\Phi_1, \Phi_2$ be mutually transitive actions on $M$, and suppose the momentum maps $\on{J}_1, \on{J}_2$ have constant rank. Then 
$\on{J}_1(M), \on{J}_2(M)$ can be given smooth structures such that $\on{J}_1(M)\xleftarrow{\J_1}M\xrightarrow{\J_2}\on{J}_2(M)$ is a Lie-Weinstein dual pair.
\end{proposition}

\begin{proof}
Since the maps $\on{J}_i:M\rightarrow \mfg_i^*$ are equivariant, they are Poisson (\cite[Proposition 12.4.1]{MarsdenRatiu1999}), and since $\on{J}_i(M)$ is a union of symplectic leaves, this property still holds when the $\on{J}_i$ are corestricted to their images.

Since $\on{J}_1^{-1}(\on{J}_1(x)) = G_2\cdot x$, we have that
\[
\ker T_x\on{J}_1 = T_x(\on{J}_1^{-1}(\on{J}_1(x))) = \mfg_2\cdot x,
\]
the first equality being a consequence of the constant rank property (see for example the discussion on page 8 of \cite{OrtegaRatiu2004}).
Then
\[
(\ker T_x\on{J}_1)^\omega = (\mathfrak{g}_2\cdot x)^\omega = \ker T_x\on{J}_2,
\]
where the second equality is a standard result. So the dual pair condition holds.

We define a smooth structure on $\on{J}_1(M)$ as follows: let $y\in \on{J}_1(M)$, and $x\in\on{J}_1^{-1}(y)$. Since $\on{J}_1$ has constant rank, there exist local charts $(U_x,\phi_x)$ about $x$ and $(V_y, \psi_y)$ about $y$ with respect to which $\on{J}_1$ takes the form of a projection, i.e., 
\begin{equation}\label{eqn:J1proj1}
\psi_y\circ\on{J}_1\circ\phi_x^{-1}(a_1,\ldots, a_m) = (a_1,\ldots, a_k, 0,\ldots,0),
\end{equation}
with $k$ independent of $x, y$. The first $k$ components of $\psi_y$, restricted to $W_y = \on{J}_1(U_x)$, defines a local coordinate chart $\eta_y:W_y\to \mathbb{R}^k$ about $y$. To show any two such charts are compatible, consider charts $(W_y, \eta_y), (W_{y'},\eta_{y'})$, with $W_y\cap W_{y'} = \on{J}_1(U_x)\cap \on{J}_1(U_{x'}) \ne \emptyset$. By constructing a shifted chart $((\Phi_2)_{g_2}(U_x), \phi_x\circ (\Phi_2)_{g_2^{-1}})$ if necessary, we can without loss of generality assume $U_x\cap U_{x'} \ne \emptyset$. 
From \eqref{eqn:J1proj1}, the level sets of $\on{J}_1\vert_{U_x\cap U_{x'}}$ are expressed as $(a_1, \ldots, a_k) = \mathrm{const.}$ and $(a_1', \ldots, a_k') = \mathrm{const.}$ in respective local coordinates, and so the first $k$ components of the (smooth) transition function $\phi_{x'}\circ \phi_x^{-1}$ only depend on the coordinates $(a_1,\ldots, a_k)$.
%\[
%\phi_{x'}\circ\phi_x^{-1}(a_1,\ldots a_m) = (f_1(a_1,\ldots, a_k), \ldots, f_k(a_1,\ldots, a_k), f_{k+1}(a_1,\ldots, a_m), \ldots, f_m(a_1,\ldots, a_m)),
%\]
%for smooth functions $f_1, \ldots, f_m$.
From
\begin{equation}\label{eqn:J1proj2}
\eta_y \circ \on{J}_1\circ\phi_x^{-1}(a_1,\ldots,a_m) = (a_1,\ldots, a_k), \qquad \eta_{y'} \circ \on{J}_1\circ\phi_{x'}^{-1}(a'_1,\ldots,a'_m) = (a'_1,\ldots, a'_k),
\end{equation}
we deduce that 
\[
\eta_{y'}\circ \eta_y^{-1}(a_1,\ldots, a_k) = (\eta_{y'}\circ \on{J}_1\circ \phi_{x'}^{-1})(\phi_{x'}\circ\phi_x^{-1}(a_1,\ldots, a_k)) =(f_1(a_1,\ldots, a_k), \ldots, f_k(a_1,\ldots, a_k))
\]
for smooth functions $f_1,\ldots, f_k$.
From either of equations \eqref{eqn:J1proj2}, $\on{J}_1:M\to \on{J}_1(M)$ is a surjective submersion with respect to this smooth structure. A similar argument holds for $\on{J}_2:M\to \on{J}_2(M)$.
\end{proof}

\begin{remark}
In general, the image of a constant rank map can exhibit so-called \emph{multiple points}, i.e., points where the tangent space to the image cannot be defined consistently\textemdash see \cite[Appendix 1, Section 1.8]{LibermannMarle1987} for a discussion. The latter part of the proof of Proposition \ref{prop:dualpairrank} essentially shows that as a consequence of the fact that level sets of $\on{J}_1$ are $G_2$-orbits, such multiple points do not exist for $\on{J}_1(M)$.
\end{remark}

\begin{remark}
The smooth structures on $\on{J}_i(M)$ in Proposition \ref{prop:dualpairrank} are necessarily unique \cite[Theorem 4.31]{Lee2013}, and the $\on{J}_i(M)$ are immersed submanifolds of $\mfg_i^*$.
\end{remark}

\begin{remark}\label{remark:constantrank}
Careful examination of the  proof of Proposition \ref{prop:dualpairrank} shows that it is sufficient to know that \emph{one} of the momentum maps has constant rank. From this, the Lie-Weinstein condition $(\ker T\on{J}_1)^\omega = \ker T\on{J}_2$ can be deduced, from which it follows that the other momentum map also has constant rank.
\end{remark}

\begin{example}
The constant rank condition on $\on{J}_1, \on{J}_2$ is necessary for proving that mutual transitivity implies the Lie-Weinstein condition. For an example where the Lie-Weinstein condition fails to hold, consider $\mathbb{R}^2$ with its usual symplectic structure, and $G_1=G_2=\on{SO}(2)$ with its usual action on $\mathbb{R}^2$. This action is Hamiltonian, with momentum map 
$\on{J}(x,y) = \frac{1}{2}(x^2+y^2)$ (on identifying $\mathfrak{so}(2)$ with $\mathbb{R}$). The action trivially commutes with itself, and the $\on{SO}(2)$-orbits agree with the level sets of the momentum map. However 
\[
\ker\,T_{(x,y)}\on{J} = 
\begin{cases} 
\mathbb{R} (y, -x) & (x,y)\ne (0,0) \\
\mathbb{R}^2 & (x,y) = (0,0)
\end{cases},
\]
and so
\[ 
\left(\ker\, T_{(0,0)}\on{J}\right)^\omega = \lbrace(0,0)\rbrace \ne \mathbb{R}^2 =  \ker\, T_{(0,0)}\on{J}.
\]
So the Lie-Weinstein condition fails to hold at the origin. 
\end{example}

For pairs of group actions, the Lie-Weinstein condition is closely related with the notion of \emph{mutually completely orthogonal} actions\textemdash see \cite{LibermannMarle1987} for further details.

We conclude with a standard useful criterion for deducing that a momentum map has constant rank.

\begin{lemma}{\rm\cite[Corollary 4.5.13]{OrtegaRatiu2004}}\label{lemma:free}
If $\Phi$ is a (locally) free Hamiltonian $G$-action, then $\on{J}$ is a submersion, and in particular has constant rank.
\end{lemma}

%\begin{proof}
%For $x\in M$, let $(\on{im}T_x\on{J})^\circ$ denote the subspace of $\mfg\simeq T_{\on{J}(x)} \mfg$ annihilated by $\on{im}T_x\on{J}$. We have
%\begin{align*}
%(\on{im} T_x\on{J})^\circ &= \lbrace \xi\in\mfg\,|\,\langle T_x\on{J}(v_x),\xi\rangle = 0 \  \textrm{for all }v_x\in T_xM \rangle \\
%&= \lbrace \xi\in\mfg\,|\, (dJ(\xi))_x(v_x) = 0 \ \textrm{for all }v_x\in T_xM\rbrace \\
%&= \lbrace \xi\in\mfg\,|\, \omega_x(\xi\cdot x, v_x) = 0\ \textrm{for all }v_x\in T_xM\rbrace \\
%&= \lbrace \xi\in\mfg\,|\, \xi\cdot x = 0\rbrace \qquad\textrm{by non-degeneracy of }\omega \\
%&= \mfg_x.
%\end{align*}
%So if $\mfg_x=\lbrace{0\rbrace}$, it follows that $\on{im} T_x\on{J} = \mfg^*$, i.e., $x$ is a regular point of $\on{J}$.

%Since $x\in M$ was arbitrary, we have that $\on{J}$ has constant rank $\dim\mathfrak{g}^*$ on $M$.
%\end{proof}

\section{The \texorpdfstring{$(\Un,\Um)$}{(U(n),U(m))} actions on \texorpdfstring{$\MnmC$}{M(nxm,C)}}\label{s3}

Following \cite{BalleierWurzbacher2012}, in this section we consider the natural Hamiltonian actions of $\Un$ and $\Um$ on $\MnmC$. We show that these actions are mutually transitive, and consequently deduce the coadjoint orbit and reduced space correspondences (Corollary \ref{coro} and Proposition \ref{prop:redcoad}). We note that Balleier and Wurzbacher instead derive these properties as a consequence of the \emph{symplectic Howe condition} \cite[Definition 2.4]{BalleierWurzbacher2012} on the actions\textemdash see \cite[Section 5.1]{BalleierWurzbacher2012} for details.

In what follows, we view elements of $\MnmC$ either as matrices or as linear maps from $\mathbb{C}^m$ to $\mathbb{C}^n$, depending on context.

\subsection{Commuting Hamiltonian actions}

First, note that $\MnmC$ is a complex inner product space, with Hermitian inner product
\[
(E,F) = \Tr(E^\dagger F).
\]
The imaginary part of this inner product defines a linear symplectic form
\[
\Omega(E,F) := \on{Im}\Tr(E^\dagger F) = \frac{1}{2i}\Tr(E^\dagger F - F^\dagger E),
\]
and $\MnmC$ is a linear K\"ahler space, with obvious complex structure.
%Using the canonical isomorphism $T_E\MnmC\simeq \MnmC$, the pair $(\MnmC, \Omega)$ can be thought of as a sympectic (in fact K\"ahler) \emph{manifold}. Explicitly, if $X_E = \frac{d}{dt}\big\vert_{t=0}(E+tX)$ for $X\in\MnmC$, then
%\[
%\Omega_E(X_E,Y_E) := \Omega(X,Y).
%\]
%
%Let $L_U$ (resp. $R_V$) denote left multiplication by $U\in \Un$ (resp. right multiplication by $V\in\Um$) on the space of matrices $\MnmC$. Using the expressions 
%$T_E L_U(X_E) = (UX)_{UE}$ and $T_E R_V(X_E) = (XV)_{EV}$ for the differentials of these actions, it is straightforward to check that $L$ and $R$ define left and right symplectic actions respectively, 
%\[
%L_U^*\Omega = \Omega \qquad\textrm{and}\qquad R_V^*\Omega = \Omega.
%\]
It is straightforward to show that the natural left $\Un$- and right $\Um$-actions act symplectically on $\MnmC$, considered as a symplectic manifold.
In fact, these actions are Hamiltonian, and we can easily compute corresponding momentum maps.

\begin{proposition}\phantomsection\label{prop:unummomentummaps}
\begin{enumerate}[{\rm(i)}]
\item
A momentum map $\on{J}_L:\MnmC\rightarrow \un^*$ corresponding to the left $\Un$-action is given by
\[
\langle \mathrm{J}_L(E), \zeta\rangle = \frac{1}{2}\Omega(\zeta E,E).
\]
\item
A momentum map $\on{J}_R:\MnmC\rightarrow \um^*$ corresponding to the right $\Um$-action is given by
\[
\langle \on{J}_R(E),\xi \rangle = \frac{1}{2}\Omega(E\xi,E).
\]
\end{enumerate}
\end{proposition}

\begin{proof}
Both results follow from the general expression for the momentum map of a linear symplectic action on a symplectic vector space\textemdash see for example \cite[Section 12.4, Example (a)]{MarsdenRatiu1999}.
\end{proof}

\begin{remark}
Both momentum maps are easily seen to be equivariant,
\[
\on{J}_L(UE) = \Ad_{U^{-1}}^*(\on{J}_L(E)) \qquad\textrm{and}\qquad \on{J}_R(EV) = \Ad_V^*(\on{J}_R(E)),
\]
hence Poisson with respect to the (+) 
Lie-Poisson structure on $\un^*$, respectively (--)
Lie-Poisson structure on $\um^*$.
\end{remark}

\subsection{Lie algebra-valued momentum maps}

Given a (real) Lie subalgebra $\mathfrak{g}\subset\mathfrak{gl}(N,\mathbb{C})$, we define the trace form $\llangle\cdot,\cdot\rrangle:\mathfrak{g}\times\mathfrak{g}\rightarrow \mathbb{R}$ by
\begin{equation}\label{traceform}
\llangle \xi, \zeta \rrangle = \on{Re} \Tr(\xi\zeta).
\end{equation}
If $\mathfrak{g}$ is invariant under conjugate transpose, then $\llangle\cdot,\cdot\rrangle$ is non-degenerate, since 
\[
\llangle \xi,\xi^\dagger\rrangle = \on{Re}\, \on{Tr}(\xi\xi^\dagger)  >0 \textrm{ for }\xi\ne 0.
\]
We can use the non-degeneracy of $\llangle\cdot,\cdot\rrangle$ to translate the momentum maps $\on{J}_L$ and $\on{J}_R$ from Proposition \ref{prop:unummomentummaps} into Lie algebra-valued momentum maps. We note in particular that if $\mfg\subset \gl(N,\mathbb{C})$ integrates to $G\subset \GL(n,\mathbb{C})$, then the identification $\mfg^*\simeq \mfg$ provided by the trace form is $G$-equivariant.

\begin{proposition}\phantomsection\label{prop:lamomUnUm}
\begin{enumerate}[{\rm (i)}]
\item
The Lie algebra-valued momentum map $\on{j}_L:\MnmC\rightarrow\un$ is 
\[
\on{j}_L(E) = \frac{i}{2}EE^\dagger.
\]
\item
The Lie algebra-valued momentum map $\on{j}_R:\MnmC\rightarrow\um$ is
\[
\on{j}_R(E) = \frac{i}{2}E^\dagger E.
\]
\end{enumerate}
\end{proposition}

\begin{proof}
(i) For $\zeta\in\un$, $E\in\MnmC$, 
\begin{align*}
\langle \on{J}_L(E),\zeta\rangle &= \frac{1}{2}\Omega(\zeta E,E) = \frac{1}{2} \on{Im} \Tr (E^\dagger \zeta^\dagger E) = -\frac{1}{2} \on{Im}\Tr(EE^\dagger \zeta) \qquad\textrm{using }\zeta^\dagger = -\zeta\\
&= \frac{1}{2} \on{Re}\Tr (iEE^\dagger\zeta) = \left\llangle \frac{i}{2}EE^\dagger,\zeta\right\rrangle.
\end{align*}
Since $\frac{i}{2}EE^\dagger \in \un$, 
%the latter equals $\left\llangle \frac{i}{2}EE^\dagger,\zeta \right\rrangle$, and 
the result follows.

\noindent (ii)
Similar.
\end{proof}

\subsection{The mutually transitive property}

\begin{proposition}\phantomsection\label{prop:UnUm}
\begin{enumerate}[{\rm (i)}]
\item
$\Un$ acts transitively on the level sets of $\on{j}_R$.
\item
$\Um$ acts transitively on the level sets of $\on{j}_L$.
\end{enumerate}
\end{proposition}

\begin{proof}
%\begin{enumerate}[(i)]
%\item
(i) From $\on{j}_R(E) = \frac{i}{2}E^\dagger E$, it is clear that the level sets of $\on{j}_R$ are invariant under the left $\Un$-action. 

Now suppose $\on{j}_R (E) = \on{j}_R(E')$, implying $E^\dagger E = (E')^\dagger E'$. Let $E_a$ denote the $a$th column of $E$, considered as a vector in $\mathbb{C}^n$. So we have the $m^2$ conditions 
\begin{equation}\label{mmconditions}
E_a^\dagger E_b = (E_a')^\dagger E'_b \qquad a,b=1,\ldots,m.
\end{equation}
The set $\lbrace E_1, E_2,\ldots, E_m\rbrace\subset \mathbb{R}^n$ has a maximal linearly independent subset $\lbrace E_{a_1}, \ldots, E_{a_k}\rbrace$ for some $k\le m$, and such a subset constitutes a basis for the subspace $\on{im} E\subset \mathbb{C}^n$. We claim that $\lbrace E'_{a_1},\ldots, E'_{a_k}\rbrace$ is a basis for $\on{im} E'$. 

Firstly, suppose $\sum_{i=1}^k \alpha_i E'_{a_i} = 0$ for some $\alpha_i\in\mathbb{C}$. Then for any $c=1,\ldots, m$,
\[
0 = (E'_c)^\dagger \left(\sum_{i=1}^k\alpha_i E'_{a_i} \right) = E_c^\dagger \left(\sum_{i=1}^k\alpha_iE_{a_i}\right),
\]
using the conditions (\ref{mmconditions}). It follows that $\sum_{i=1}^k \alpha_i E_{a_i} \in \on{im}E \cap (\on{im}E)^\perp = \lbrace 0 \rbrace$ (where $^\perp$ denotes orthogonality with respect to the usual inner product in $\mathbb{C}^n$). Hence $\sum_{i=1}^k \alpha_i E_{a_i} = 0$, and so linear independence of the $E_{a_i}$ guarantees that $\alpha_i=0$ for all $i=1,\ldots, k$, proving linear independence of $\lbrace E'_{a_1},\ldots, E'_{a_k}\rbrace$.

Also, for any $c=1,\ldots, m$, there exist $\beta_i \in \mathbb{C}$ such that $E_c = \sum_{i=1}^k \beta_i E_{a_i}$. Then
\[
(E'_d)^\dagger \left(E'_c - \sum_{i=1}^k \beta_i E'_{a_i} \right) = E_d^\dagger \left(E_c - \sum_{i=1}^k \beta_i E_{a_i}\right) = 0\qquad d=1,\ldots, m,
\]
implying that $E_c' - \sum_{i=1}^k\beta_i E_{a_i}' \in \on{im}E' \cap (\on{im}E')^\perp = \lbrace 0\rbrace$, i.e., $E_c' = \sum_{i=1}^k\beta_iE_{a_i}'$. Hence $\lbrace E_{a_1}',\ldots, E_{a_k}'\rbrace$ span $\on{im}E'$.

Now define $U:\on{im}E\rightarrow \on{im} E'$ by $U(E_{a_i}) := E_{a_i}'$. From \eqref{mmconditions} we see $U$ is an isometry. It can be extended to the entire space $\mathbb{C}^{n}$ by picking an arbitrary isometry $(\on{im}E)^\perp \rightarrow (\on{im}E')^\perp$, giving $U\in\Un$.

From the discussion above, we see that if $E_c = \sum_{i=1}^k\beta_i E_{a_i}$, then $E_c' = \sum_{i=1}^k\beta_iE_{a_i}'$. It follows that
\[
U(E_c) = \sum_{i=1}^k\beta_i U(E_{a_i}) = \sum_{i=1}^k\beta_i E_{a_i}' = E_c'
\]
for all $c=1,\ldots, m$, and so $E' = UE$. Hence $E$ and $E'$ lie in the same $\Un$-orbit.
\medskip
%\item

\noindent (ii) Same method as part (i), except applied to rows of $E$ instead of columns.
%\end{enumerate}
\end{proof}

We have proved mutual transitivity of the $(\Un,\Um)$ actions on $\MnmC$. Thus we get a (generalised) dual pair of momentum maps 
%\[
%\un_{\on{j}_L} \stackrel{\on{j}_L}{\longleftarrow} \MnmC \stackrel{\on{j}_R}{\longrightarrow} \um_{\on{j}_R}
%\]
\begin{equation}\label{diag:gendpUnUm}
\begin{diagram}
&&  \MnmC\\
&\ldTo^{\mathrm{j}_{L}} && \rdTo^{\mathrm{j}_{R}} \\
\un_{\on{j}_L} &&&&\um_{\on{j}_R}
\end{diagram}
\end{equation}
where $\un_{\on{j}_L}$ and $\um_{\on{j}_R}$ are the images of the left and right momentum maps respectively.

\begin{remark}
The momentum maps $\on{j}_L, \on{j}_R$ in fact define a singular dual pair, in the sense of Ortega \cite{Ortega2003, OrtegaRatiu2004}.
\end{remark}

\subsection{Adjoint orbit correspondence}
We briefly recall the description of the adjoint orbit correspondence from \cite{BalleierWurzbacher2012}. Assuming for concreteness that $n\ge m$, any $E\in \MnmC$ has a unique singular-value decomposition $E = U\Sigma V^\dagger$, where $U\in \Un, V\in \Um$, and
\[
\Sigma = \begin{bmatrix} 
\sigma_1 & 0 & \ldots & 0\\
0 & \sigma_2 & \ldots & 0\\
\vdots & \vdots & \ddots & \vdots \\
0 & 0 & \ldots & \sigma_m \\
0 & 0 & \ldots & 0 \\
\vdots & \vdots & \vdots & \vdots
\end{bmatrix}
\]
with $\sigma_1 \ge \sigma_2\ge \ldots \ge \sigma_m \ge 0$. %Denote the diagonal element of $\mathfrak{u}(N)$ with entries $i(\lambda_1,\lambda_2,\ldots, \lambda_N)$ by $[\lambda_1,\lambda_2,\ldots, \lambda_N]$. 
The expressions for the momentum maps (Proposition \ref{prop:lamomUnUm}) imply that $\on{j}_L(E)$ is in the  adjoint orbit of the diagonal matrix $\on{diag}[\frac{i}{2}\sigma_1^2, \frac{i}{2}\sigma_2^2, \ldots ,\frac{i}{2}\sigma_m^2,0,\ldots, 0]\in\un$, while $\on{j}_R(E)$ is in the adjoint orbit of $\on{diag}[\frac{i}{2}\sigma_1^2, \frac{i}{2}\sigma_2^2, \ldots ,\frac{i}{2}\sigma_m^2]\in\um$. The correspondence between such orbits, for all $\sigma_1\ge \sigma_2\ge\ldots \ge \sigma_m \ge 0$, is one-to-one (note our conventions for $\on{j}_R$ introduce a minus sign relative to \cite{BalleierWurzbacher2012}).

\subsection{Restriction to a Lie-Weinstein dual pair}

For completeness, we now characterise the subset of $\MnmC$ where the generalised dual pair 
\eqref{diag:gendpUnUm} becomes a Lie-Weinstein dual pair. As before, assume $n\ge m$ for concreteness.

\begin{proposition}
The momentum maps $\on{j}_L, \on{j}_R$ define a Lie-Weinstein dual pair on the (open) subset $\MnmCm$ of full rank matrices in $\MnmC$. 
\end{proposition}

\begin{proof}
The right $\Um$-action is free on $\MnmCm$. Then using Lemma \ref{lemma:free}, $\on{j}_R$ has constant rank there, and then Remark \ref{remark:constantrank} implies the result.
\end{proof}

\begin{proposition}
The set $\MnmCm$ is the largest subset of $\MnmC$ on which $\on{j}_L, \on{j}_R$ define a Lie-Weinstein dual pair. 
\end{proposition}

\begin{proof}

Let $E\in \MnmC$ have singular-value decomposition $U\Sigma V^\dagger$, where $\Sigma$ is as described in the previous section. Suppose $\sigma_{m-k}$ is the last non-zero $\sigma_i$ (implying $\sigma_{m-k+1} = \ldots =\sigma_m = 0$). Note $k=0$ is possible. From
\[
T_E\on{j}_R(X_E) = \frac{i}{2}(X^\dagger E + E^\dagger X),
\]
we see that $\ker T_E\on{j}_R$ consists of matrices $X = U\widetilde{X}$, where $\widetilde{X}\in\MnmC$ is non-zero only in the lower $m\times(n-m+k)$ block. Hence $\on{im} T_E\on{j}_R = nm - m(n-m+k) = m(m-k)$.
This equals $\dim \um = m^2$ iff all of the $\sigma_i$ are non-zero, which occurs iff $E$ has full rank $m$.
\end{proof}

\section{Matrix analogue of ideal fluid dual pair}\label{s4}

In this section, we describe a symplectic structure on $\MtnmR$ and demonstrate that the left (resp. right) action of $\Sptn$ (resp. $\Om$) is Hamiltonian. We then show that on a suitable subset of $\MtnmR$, the $\Sptn$- and $\Om$-actions are mutually transitive, and deduce that they define a Lie-Weinstein dual pair. Finally, we describe explicitly the correspondence between adjoint orbits in the images of the respective momentum maps. 

This dual pair was originally discussed in \cite[pp.502-506]{KazhdanKostantSternberg1978}.

\subsection{Commuting Hamiltonian actions}
The vector space $\MtnmR$ has a symplectic form
\[
\Omega(E,F) := \Tr(E^\top\mathbb{J}F),
\]
where $\mathbb{J} = \begin{bmatrix} 0_n & I_n \\ -I_n & 0_n\end{bmatrix}$. As before, we think of the pair $(\MtnmR, \Omega)$ as a symplectic manifold by using the canonical isomorphism $T_E\MtnmR \simeq \MtnmR$.

%Let $L_S$ (resp. $R_O$) denote left multiplication by $S\in\Sptn$ (resp. right multiplication by $O\in\Om$). Again as before, it is straightforward to check that $L$ and $R$ define left and right symplectic actions respectively,
%\[
%L_S^*\Omega = \Omega \qquad \textrm{and}\qquad R_O^*\Omega = \Omega.
%\]
The natural left $\Sptn$- and right $\Om$-actions act symplectically on $\MtnmR$, considered as a symplectic manifold.
These actions are Hamiltonian, with momentum maps $\on{J}_L:\MtnmR\rightarrow \sptn^*$ and $\on{J}_R:\MtnmR\rightarrow \om^*$ given by
\begin{equation}\label{lere}
\langle\on{J}_L(E),\zeta\rangle = \frac{1}{2}\Omega(\zeta E,E) \qquad\textrm{and}\qquad \langle\on{J}_R(E),\xi\rangle = \frac{1}{2}\Omega(E\xi,E).
\end{equation}
Again, both momentum maps are equivariant,
\[
\on{J}_L(SE) = \Ad_{S^{-1}}^*(\on{J}_L(E)) \qquad\textrm{and}\qquad \on{J}_R(EO) = \Ad_O^*(\on{J}_R(E)).
\]

\subsection{Lie algebra-valued momentum maps}
We recall the trace form (\ref{traceform})
\[
\llangle \xi,\zeta\rrangle = \on{Re}\on{Tr}(\xi\zeta) = \on{Tr}(\xi\zeta),
\] 
now defined on the Lie algebras $\sptn$ and $\om$. Using $\llangle\cdot,\cdot\rrangle$ to identify Lie algebras with their duals, we can again define Lie algebra-valued momentum maps.

\begin{proposition}
\begin{enumerate}[(i)]
\item 
The Lie algebra-valued momentum map $\mathrm{j}_L:\MtnmR\rightarrow\sptn$ is 
\[
\mathrm{j}_L(E) = -\frac{1}{2}EE^\top\mathbb{J}.
\]
\item
The Lie algebra-valued momentum map $\mathrm{j}_R:\MtnmR\rightarrow\om$ is
\[
\mathrm{j}_R(E) = -\frac{1}{2}E^\top\mathbb{J} E.
\]
\end{enumerate}
\end{proposition}

\begin{proof}
(i) For $\zeta\in\sptn,\ E\in\MtnmR$,
\begin{align*}
\langle\mathrm{J}_L(E),\zeta\rangle &= \frac{1}{2}\Omega(\zeta E,E) = \frac{1}{2}\mathrm{Tr}(E^\top\zeta^\top\mathbb{J} E)  \qquad\textrm{by \eqref{lere} }\\
&= \frac{1}{2}\mathrm{Tr}(EE^\top\zeta^\top \mathbb{J}) = \frac{1}{2}\mathrm{Tr}(-EE^\top\mathbb{J}\zeta) \qquad\textrm{since }\zeta\in\sptn \\
&= \left\llangle -\frac{1}{2}EE^\top\mathbb{J}, \zeta\right\rrangle.
\end{align*}
Since $-\frac{1}{2}EE^\top\mathbb{J}\in\sptn$, the result follows.

\noindent (ii)
Similar.
\end{proof}

%%%%

\subsection{The mutually transitive property on full rank matrices}\label{sec:mtSpO}

In contrast with the case of the $\Un$- and $\Um$-actions on $\MnmC$, 
demonstration of the mutually transitive property of the $\Sptn$- and $\Om$-actions requires restriction to a subset of $\MtnmR$. To this end, let $\MtnmRm\subset \MtnmR$ denote the matrices of rank $m$. In order for $\MtnmRm$ to be nonempty, we require $m\le 2n$. Defining $ f:\MtnmR\rightarrow \mathbb{R}$ by $f(E) = \det(E^\top E)$,
we see that $\MtnmRm = f^{-1}((0,\infty))$, and so $\MtnmRm$ is an open subset of $\MtnmR$. It follows that $\Omega$ remains non-degenerate when restricted to $\MtnmRm$. Additionally, since elements of $\Om$ and $\Sptn$ have full rank, their group actions preserve $\MtnmRm$. We denote restrictions of $\Omega,\, \on{j}_L,\, \on{j}_R$ to $\MtnmRm$ by the same symbols for convenience.

\begin{proposition}\phantomsection\label{prop:trans}
\begin{enumerate}[{\rm (i)}]
\item
$\Sptn$ acts transitively on the level sets of $\,\on{j}_R:\MtnmRm\rightarrow \om$.
\item
$\Om$ acts transitively on the level sets of $\,\on{j}_L:\MtnmRm\rightarrow \sptn$.
\end{enumerate}
\end{proposition}

Before proving Proposition (\ref{prop:trans}), we need the following standard result.

\begin{proposition}[Witt's theorem]{\rm\cite[Theorem 3.9]{Artin1957}}
Let $V$ be a finite-dimensional vector space, over a field $\mathbb{F}$ of characteristic different from 2, and $q:V\times V\rightarrow \mathbb{F}$ a symmetric or anti-symmetric nondegenerate bilinear form on $V$. If $f:U\rightarrow U'$ is a (linear) isometry between two subspaces of $V$, then $f$ extends to an isometry of $V$.
\end{proposition} 

\begin{proof}[Proof of Proposition (\ref{prop:trans})]
%\begin{enumerate}[(i)]
%\item
(i) Since $\on{j}_R(E) = -\frac{1}{2}E^\top\mathbb{J} E$, clearly the left $\Sptn$-action preserves the level sets of $\on{j}_R$.

Now suppose $\mathrm{j}_R(E) = \mathrm{j}_R(E')$, implying $E^\top \mathbb{J}E = (E')^\top\mathbb{J} E'$. Letting $E_a$ denote the $a$th column of $E$, considered as a vector in $\mathbb{R}^{2n}$, this gives the $m^2$ conditions
\[
E_a^\top\mathbb{J} E_b = (E'_a)^\top\mathbb{J} E'_b \qquad\textrm{for }a,b=1,\ldots, m.
\]
Define $S:\on{im}E\rightarrow \on{im}{E'}$ by $S(E_a) = E'_a$ (this is well-defined, since the columns $E_a$ are linearly independent). So the above condition becomes
\[
\omega(E_a,E_b) = \omega(SE_a,SE_b),
\]
where $\omega(X,Y) := X^\top\mathbb{J} Y$ denotes the standard symplectic form on $\mathbb{R}^{2n}$.
By Witt's theorem, there exists a linear extension $S:\mathbb{R}^{2n}\rightarrow\mathbb{R}^{2n}$ preserving $\omega$. Then $S\in\Sptn$, and  $E' = SE$. So $E'$ and $E$ lie in the same $\Sptn$-orbit.

\medskip
%\item
\noindent (ii) Since $\on{j}_L(E) = -\frac{1}{2}EE^\top \mathbb{J}$, clearly the right $\Om$-action preserves the level sets of $\on{j}_L$.

Now suppose $\mathrm{j}_L(E) = \mathrm{j}_L(E')$, implying $EE^\top = E'(E')^\top$. This can be put into a form similar to Proposition (\ref{prop:UnUm})(i) by letting $F = E^\top, F'=(E')^\top$. Following a similar argument as there, we obtain an isometry $O:\mathbb{R}^m\rightarrow \mathbb{R}^m$ with $F' = OF$, i.e., $E' = E O^\top$. Since $O^\top\in \Om$, we see that $E$ and $E'$ are related by the right $\Om$ action.

%\end{enumerate}
\end{proof}

We have proved mutual transitivity of the 
$(\Sptn,\Om)$ actions. Since $\Om$ acts freely on $\MtnmRm$, we conclude by Lemma \ref{lemma:free} that $\on{j}_R$, and so also $\on{j}_L$ (Remark \ref{remark:constantrank}), has constant rank on $\MtnmRm$, and so the momentum maps define a Lie-Weinstein dual pair
\begin{diagram}
&& \MtnmRm\\
&\ldTo^{\on{j}_{L}} && \rdTo^{\on{j}_{R}} \\
\sptn_{\on{j}_L} &&&&\om_{\on{j}_R}
\end{diagram}
where $\sptn_{\on{j}_L}$ and $\om_{\on{j}_R}$
are the images of the left and right momentum maps respectively.

\begin{remark}
For $m<2n$, the $\Sptn$-action has non-compact isotropy group at points of $\MtnmRm$. Hence it cannot be proper \cite[Proposition 2.3.8 (i)]{OrtegaRatiu2004}.
\end{remark}

\subsection{Adjoint orbit correspondence}\label{section:adjointcorrespondence}

By Corollary \ref{coro} there is a one-to-one correspondence
between coadjoint orbits in the images $\sptn^*_{\on{J}_L}$ and $\om^*_{\on{J}_R}$.
Equivalently, since $\llangle\cdot,\cdot\rrangle$ is $\Ad$-invariant, we have a correspondence between
adjoint orbits in $\sptn_{\on{j}_L}$ and  $\om_{\on{j}_R}$.

From \cite{Xu2003} we know  that every matrix $E\in\MtnmR$ of rank $m$
has an singular-value-decomposition-like representation as $E=SDO$ with $S\in\Sptn$, $O\in\Om$,
and $D$ given by

\begin{equation*}
D = \begin{blockarray}{cccc}
\begin{block}{*{3}{>{\scriptstyle}c}c}
p&q&p \\
\end{block}
\begin{block}{[ccc] >{\scriptstyle}c} 
\Sigma&0&0&p \\ 0&I&0&q \\ 0&0&0&r \\ 0&0&\Sigma&p \\ 0&0&0&q \\ 0&0&0&r \\
\end{block}
\end{blockarray},
\end{equation*}
where $\Sigma$ is a diagonal block with positive entries $\sigma_1,\dots,\sigma_p$.
Here $q=m-2p$ is imposed by the rank condition, and $r=n-p-q = n-m+p$.
Since $\on{j}_R(E)$ is $\Om$-conjugate to
\begin{equation*}
\on{j}_R(D) = -\frac{1}{2}D^\top\mathbb{J}D
= -\frac{1}{2}\begin{blockarray}{cccc} 
\begin{block}{*{3}{>{\scriptstyle}c}c}
p&q&p\\
\end{block}
\begin{block}{[ccc] >{\scriptstyle}c}
0&0&\Sigma^2&p \\ 0&0&0&q \\ -\Sigma^2&0&0&p \\
\end{block}
\end{blockarray}
\in\om,
\end{equation*}
we conclude that the image of $\on{j}_R$ consists of the adjoint orbits of $\Om$
that correspond to normal forms that are block diagonal, with entries
$-\frac{1}{2}\begin{bmatrix} 0& \sigma_1^2 \\ -\sigma_1^2 & 0\end{bmatrix},
\dots,-\frac{1}{2}\begin{bmatrix} 0& \sigma_p^2 \\ -\sigma_p^2 & 0\end{bmatrix}$, and a $q\times q$ zero block. On the other hand $\on{j}_L(E)$ is $\Sptn$-conjugate to
\begin{equation*}
\on{j}_L(D) =-\frac{1}{2}DD^\top \mathbb{J}
= -\frac{1}{2}\begin{blockarray}{ccccccc}
\begin{block}{*{6}{>{\scriptstyle}c}c}
p&q&r&p&q&r\\
\end{block}
\begin{block}{[cccccc] >{\scriptstyle}c}
0&0&0&\Sigma^2&0&0&p \\ 0&0&0&0&I&0&q \\ 0&0&0&0&0&0&r \\ -\Sigma^2&0&0&0&0&0&p \\ 0&0&0&0&0&0&q \\ 0&0&0&0&0&0&r \\
\end{block}
\end{blockarray}\in\sptn,
\end{equation*}
hence the image of $\on{j}_L$ consists of the adjoint orbits of $\Sptn$
that correspond to normal forms that are block diagonal, with entries
$-\frac{1}{2}\begin{bmatrix} 0& \sigma_1^2 \\ -\sigma_1^2 & 0\end{bmatrix},
\dots, -\frac{1}{2}\begin{bmatrix} 0& \sigma_p^2 \\ -\sigma_p^2 & 0\end{bmatrix}$,
$q$ blocks of type $-\frac{1}{2}\begin{bmatrix} 0& 1 \\ 0& 0\end{bmatrix}$, and an $r\times r$ zero block. 
%$q$ times, and $\begin{bmatrix} 0& 0 \\ 0& 0\end{bmatrix}$ $n-p-q$ times.

We conclude that the adjoint orbit correspondence is between the two above mentioned orbits,
characterized by the integer $p$ and the positive values $\sigma_1,\dots,\sigma_p$.

\begin{remark}
The adjoint orbit correspondence was described in \cite[p.505]{KazhdanKostantSternberg1978} in the special case $q=0$.
\end{remark}

\subsection{Relations between the \texorpdfstring{$(\Un, \Um)$ and $(\Sptn,\Om)$}{(U(n),U(m)) and (Sp(2n,R), O(m))} momentum maps}

As symplectic manifolds, $(\MnmC,\Omega_\mathbb{C})$ and $(\MtnmR,\Omega_\mathbb{R})$ are 
isomorphic, where we now use obvious notation to distinguish between symplectic forms. In fact, under the identification
\[
E_\CC = E_1+iE_2 \in \MnmC \longleftrightarrow E_\RR = \begin{bmatrix} E_1 \\ E_2 \end{bmatrix} \in \MtnmR 
\]
we see that
\[
\Omega_\mathbb{C}(E_\CC,F_\CC) = \on{Im} \Tr(E_\CC^\dagger F_\CC) =   \Tr (E_1^\top F_2 - E_2^\top F_1) = \Omega_\mathbb{R}( E_\RR, F_\RR ).
\]

We can realize $\un$ as a Lie subalgebra of $\sp(2n,\RR)$
with the map 
$$\ell:\zeta_1+i\zeta_2\in\un\mapsto\begin{bmatrix} \zeta_1& -\zeta_2 \\ \zeta_2 & \zeta_1\end{bmatrix}
\in\sp(2n,\RR)$$
(noting that $\zeta_1^\top = -\zeta_1, \zeta_2^\top = \zeta_2$).
Denoting by $i$ the inclusion of $\om$ into $\um$, we obtain:

\begin{proposition}The diagram
\[
\xymatrix{
\sp(2n,\RR)^*\ar[d]_{\ell^*} &\MtnmR \ar[l]_{\J_{\Sptn}}\ar[r]^{\J_{\Om}} \ar[d]_{=}&
{\mathfrak{o}(m)}^*\\
\un^*& \,\,\MnmC\,\,\ar[l]_{{\J}_{\Un}} \ar[r]^{{\J}_{\Um}} &
\um^*\ar[u]_{i^*}\\
}
\]
commutes, where here momentum maps are labelled by their corresponding groups.
\end{proposition}

\begin{proof}
Let $E_\mathbb{R}=\begin{bmatrix} E_1 \\ E_2\end{bmatrix}
\in \MtnmR$ be arbitrary.
For all $\zeta = \zeta_1+i\zeta_2\in\un$ we have
\[
\ell(\zeta)E_\mathbb{R} = \begin{bmatrix} \zeta_1& -\zeta_2 \\ \zeta_2 & \zeta_1\end{bmatrix} \begin{bmatrix} E_1 \\ E_2 \end{bmatrix} = \begin{bmatrix}\zeta_1 E_1 - \zeta_2 E_2 \\ \zeta_2 E_1 + \zeta_1 E_2 \end{bmatrix} \longleftrightarrow (\zeta_1+i\zeta_2)(E_1+iE_2) = \zeta E_\CC \in \MnmC,
\]
and so
\[
\langle\ell^*(\J_{\Sptn}(E_\RR)),\zeta \rangle
=\frac12\Omega_\RR(\ell(\zeta)E_\RR,E_\RR)=\frac12\Omega_\CC\left(\zeta E_\CC, E_\CC \right)
=\langle \J_{\Un}(E_\CC),\zeta\rangle.
\]
The identity $i^*\circ\on{J}_{\Um} = \on{J}_{\Om}$ is proved similarly.
%Also, for all $\xi\in \om$ we get
%\[
%\langle \J_{\Om}(E_\RR),\xi\rangle
%=\frac12\Omega_\RR(E_\RR\xi,E_\RR)
%= \frac{1}{2}\Omega_\CC(E_\CC\xi, E_\CC) = \langle i^*\J_{\Um}(E_\CC), \xi\rangle.
%\]
\end{proof}

%%%%%%%%%%%%%%%%%%%%%%%%
%%%%%%%%%%%%%%%%%%%%%%%%

\section{Matrix analogue of the EPDiff dual pair}\label{s5}

In this section, upon identifying $T^*\MnmR$ with $\MtnmR$, we describe the lifted cotangent action of $\GL(n,\RR)$ and $\GL(m,\RR)$ on $\MtnmR$, and demonstrate that these actions are mutually transitive, and deduce that they define a Lie-Weinstein dual pair on a suitable subset of $\MtnmR$. We then give an explicit description of the adjoint orbit correspondence. Finally, we outline the relationship between $(\GL(n,\RR),\GL(m,\RR))$ momentum maps and the $(\Sptn, \Om)$ momentum maps of the previous section.

\subsection{Commuting Hamiltonian actions}
%actions of $\GL(n,\RR)$ and $\GL(m,\RR)$}

We identify $T^*\MnmR$ with  $\MnmR\times\MnmR\simeq\MtnmR$ using the 
non-degenerate pairing on $\MnmR$ given by $( X,Y)\mapsto\Tr (X^\top Y)$.
More precisely, 
%the pair $(Q,P)$ denotes the following element of $T_Q^*\MnmR$:
\begin{equation}\label{parq}
(Q,P):X\in T_Q\MnmR \simeq \MnmR\mapsto\Tr(P^\top X).
\end{equation}
Thus the canonical symplectic form on the cotangent bundle is the same as that induced by the constant symplectic form on the linear space $\MtnmR$:
%given by the linear symplectic form on $ M_{2n\x m}(\RR)$:
\[
\Omega(X,Y)=\Tr(X^\top\mathbb{J}Y),\quad X,Y\in\MtnmR.
\]

On  $\MnmR$ we consider the left $\GL(n,\RR)$-action and the right $\GL(m,\RR)$-action,
together with their cotangent lifted action on $T^*\MnmR$:
\begin{equation}\label{gln}
A\cdot(Q,P)=(AQ,(A^\top)^{-1}P),\quad A\in\GL(n,\RR)
\end{equation}
and
\begin{equation}\label{glm}
(Q,P)\cdot B=(QB,P(B^\top)^{-1}), \quad B\in\GL(m,\RR).
\end{equation}

\begin{proposition}
The left $\GL(n,\RR)$-action and right $\GL(m,\RR)$-action on $T^*\MnmR$ are Hamiltonian with cotangent momentum maps 
%$J_L:T^*\M\to\gl(N,\RR)^* and 
\[
\langle \J_L(Q,P),\xi\rangle=\Tr(QP^\top \xi),\quad\xi\in\gl(n,\RR)
\]
and 
\[
\langle \J_R(Q,P),\eta\rangle=\Tr(P^\top Q \eta),\quad\eta\in\gl(m,\RR).
\]
\end{proposition}

\begin{proof}
Every cotangent lifted action is Hamiltonian and has an equivariant momentum map.
%Using the above expression for cotangent bundle momentum maps, we write 
The left action momentum map is
\[
\langle \J_L(Q,P),\xi\rangle=(Q,P)(\xi Q)\stackrel{\eqref{parq}}{=}\Tr(P^\top \xi Q)=\Tr(QP^\top \xi),\quad\xi\in\gl(n,\RR).
\]
Similarly,
\[
\langle \J_R(Q,P),\eta\rangle=(Q,P)(Q\eta)\stackrel{\eqref{parq}}{=}\Tr(P^\top Q \eta),\quad\eta\in\gl(m,\RR)
\]
is the cotangent bundle momentum map for the right action.
\end{proof}

When using the {trace form} $\llangle X,Y\rrangle=\Tr(XY)$
to identify $\gl(n,\RR)^*$ with $\gl(n,\RR)$, the momentum maps above take the concise expressions
\begin{equation}\label{jrjl}
\on{j}_L(Q,P)=QP^\top \in\gl(n,\RR),\quad \on{j}_R(Q,P)=P^\top Q\in\gl(m,\RR).
\end{equation}
%These momentum maps $\on{j}_L, \on{j}_R$ are clearly $\Ad$-equivariant with respect to the $\GL(n,\RR), \GL(m,\RR)$-actions respectively. 

\subsection{The mutually transitivity property on full rank matrices}\label{sec:mtGLnGLm}

Let $\MnmRm\subset\MnmR$ denote the subset of rank $m$ matrices. In the sequel we will at times identify $\MnmRm$ with linear injective maps from $\RR^m$ to $\RR^n$. 
For $\MnmRm$ to be non-empty, we require $m\le n$. 
$\MnmRm$ is an open subset of $\MnmR$, and so $\MnmRm\times\MnmRm$ is an open subset of $T^*\MnmR\simeq \MnmR\times\MnmR$. Hence the symplectic form on $T^*\MnmR$ restricts to a symplectic form on $\MnmRm\times \MnmRm$. It is clear that $\MnmRm\times\MnmRm$ is preserved by the cotangent lifted actions of  $\GL(n,\RR)$ and $\GL(m,\RR)$.

\begin{proposition}\label{lone}
The group $\GL(m,\RR)$ acts transitively on level sets of the left momentum map $\on{j}_L$ restricted to $\MnmRm\times\MnmRm$.
\end{proposition}

\begin{proof}
The cotangent $\GL(m,\RR)$-action \eqref{glm} preserves 
the fibers of the momentum map $\on{j}_L$ in \eqref{jrjl}:
\[
\on{j}_L((Q,P)\cdot B)=\on{j}_L(QB,P(B^\top)^{-1})= {QB(P(B^\top)^{-1})^\top = QP^\top} =\on{j}_L(Q,P).
\]
Suppose now that $\on{j}_L(Q,P)=\on{j}_L(Q',P')$. 
%By multiplying this identity with $F^\top$, where $F:\RR^n\to\RR^m$ is a right linear inverse for $Q$
From $QP^\top=Q'P'^\top$ we deduce that the linear injective mappings 
corresponding to $Q$ and $Q'$ have the same range, since both $P^\top$ and $P'^\top$ 
correspond to linear surjective maps $\RR^n\to\RR^m$.
Thus there exists $B\in\GL(m,\RR)$ with $Q'=QB$,
and by inserting in the above identity we get $QP^\top=QBP'^\top$.
By the injectivity of $Q$ follows $P^\top=BP'^\top$.
This ensures that $(Q',P')=(Q,P)\cdot B$.
\end{proof}

To prove the transitivity of the $\GL(n,\RR)$-action on level sets 
of the {right} momentum map {$\on{j}_R$},  we will need the fact that any two matrices in $\MnmRm\subset\MnmR$
can be completed to invertible $n\times n$ matrices by using the same matrix.

\begin{lemma}\label{ps}
Assume $m<n$.
Given matrices $Q_1,Q_2\in \MnmRm$,
%L_{inj}(\RR^m,\RR^n)$ viewed as rank $m$ matrices in $M_{m\x n}(\RR)$, 
there exists $X\in \on{M}_{n\times(n- m)}(\RR)$ such that the order $n$ square matrices
 $[Q_1\ X]$ and $[Q_2\ X]$ are invertible.
\end{lemma}

\begin{proof}
An easy induction argument on $m$ ensures that there exists a subspace $V$ of $\RR^n$ that is simultaneously a complement to  both $m$-dimensional subspaces $\on{im} Q_1$ and $\on{im}Q_2$. Then we choose a basis of $V$ and we build the matrix $X$ whose columns are
these basis vectors.
(The induction argument is based on the fact that there exists $v\in\RR^n$ 
that doesn't belong to these two $m$-dimensional subspaces, so
the subspaces $\RR v+\on{im} Q_1$ and  $\RR v+\on{im} Q_2$ both
have dimension $m+1$.)
\end{proof}

\begin{proposition}\label{ltwo}
The group $\GL(n,\RR)$ acts transitively on level sets of the right momentum map $\on{j}_R$ restricted to $\MnmRm\times\MnmRm$.
\end{proposition}

\begin{proof}
The cotangent $\GL(n,\RR)$-action \eqref{gln}
preserves the fibers of the momentum map $\on{j}_R$ in \eqref{jrjl}:
\[
\on{j}_R(A\cdot(Q,P))=\on{j}_R(AQ,(A^\top)^{-1}P)=Q^\top A^\top(A^\top)^{-1}P=Q^\top P=\on{j}_R(Q,P).
\]

Suppose now that $\on{j}_R(Q,P)=\on{j}_R(Q',P')$, i.e.,
\begin{equation}\label{initial}
(Q')^\top P'=Q^\top P.
\end{equation} 
We are looking for $A\in\GL(n,\RR)$ with properties $Q'=AQ$ and $A^\top P'=P$.
The special case $m=n$ is easy, because in this case all $Q,Q',P,P'\in\GL(n,\RR)$,
so we can put $A=Q'Q^{-1}\in\GL(n,\RR)$ which gives us $A^\top P'=(Q^{-1})^\top(Q')^\top P'=P$.

Next we consider the general case $m<n$.
Since both $P$ and $P'$ are injective, there exists a matrix $C\in\GL(n,\RR)$
such that $P=C^\top P'$. Since $C^{-1} Q'\in \MnmRm$,
%L_{inj}(\RR^m,\RR^n)$,
by Lemma \ref{ps} there exists a matrix $X\in \on{M}_{n\times(n- m)}(\RR)$ 
such that the two order $n$ matrices
$D=[Q\ X]$ and $[(C^{-1} Q')\ X]$  are both invertible. 
Putting $X'=C X$,
the order $n$ matrix $D'=[Q'\ X']=C[(C^{-1} Q')\ X]$ is invertible too.

Now the matrix 
$A=D'D^{-1}\in \GL(n,\mathbb{R})$ satisfies
$A\cdot(Q,P)=(Q',P')$.
Indeed, because
\begin{equation*}
(X')^\top P'=X^\top C^\top P'=X^\top P,
\end{equation*}
we have that
\[
A^\top P'
=(D^{-1})^\top (D')^\top P'
=(D^{-1})^\top\begin{bmatrix} (Q')^\top P' \\(X')^\top P'\end{bmatrix}
%\left(\begin{array}{ccc} (Q')^\top  \\(X')^\top \end{array}\right) 
\stackrel{\eqref{initial}}{=}(D^{-1})^\top
\begin{bmatrix} Q^\top P \\X^\top P \end{bmatrix}
{=}(D^{-1})^\top
D^\top P=P.
\]
On the other hand 
$$AQ=D'D^{-1}Q=[Q'\ X']\begin{bmatrix} I \\0 \end{bmatrix} =Q',$$
thus getting the required transitivity conditions.
\end{proof}

We have proved mutual transitivity of the $(\GL(n,\RR), \GL(m,\RR))$ actions.
By injectivity of elements of $\MnmRm$, it is straightforward to see that the right action \eqref{glm} of $\GL(m,\RR)$ on $\MnmRm\times\MnmRm$ is free. 
So by Lemma \ref{lemma:free}, $\on{j}_R$ and $\on{j}_L$ (compare Remark \ref{remark:constantrank}) have constant rank, and thus the momentum maps $\on{j}_L, \on{j}_R$ define a Lie-Weinstein dual pair
\begin{diagram}
&& \MnmRm\times\MnmRm \\
&\ldTo^{\mathrm{j}_{L}} && \rdTo^{\mathrm{j}_{R}} \\
\gl(n,\RR)_{\on{j}_L} &&&&\gl(m,\RR)_{\on{j}_R}
\end{diagram}
where $\gl(n,\RR)_{\on{j}_L}$ and $\gl(m,\RR)_{\on{j}_R}$ are the images of the left and right momentum maps respectively.

\begin{remark}
For $m<n$, the $\GL(n,\RR)$-action has non-compact isotropy group at points of $\MnmRm\times\MnmRm$. Hence it cannot be proper \cite[Proposition 2.3.8 (i)]{OrtegaRatiu2004}.
\end{remark}

\subsection{Adjoint orbit correspondence}\label{sec:adjcorr}

As in the discussion of subsection \ref{section:adjointcorrespondence}, we have a one-to-one correspondence between $\GL(n,\RR)$-orbits in the image of $\on{j}_L$ and $\GL(m,\RR)$-orbits in the image of $\on{j}_R$. 

We now characterise the images of $\on{j}_L$ and $\on{j}_R$, and the adjoint orbit correspondence between these images. Define the sets

{\begin{align*}
S_L &:= \lbrace \zeta \in \gl(n,\RR)\,|\, \rank\zeta = m\rbrace \\
S_R &:= \lbrace \xi \in \gl(m,\RR)\,|\,  \rank\xi\ge 2m-n \rbrace.
\end{align*}}

\begin{lemma}\phantomsection\label{lemma:immom}
\begin{enumerate}[(i)]
\item $\gl(n,\RR)_{\on{j}_L} \subset S_L$.
\item $\gl(m,\RR)_{\on{j}_R}\subset S_R$.
\end{enumerate}
\end{lemma}

\begin{proof}
Thinking of $(Q,P)\in\MnmRm\times\MnmRm$ as linear maps, we have that
$Q:\mathbb{R}^m\rightarrow \mathbb{R}^n$ is injective, while $P^\top:\mathbb{R}^n\rightarrow \mathbb{R}^m$ is surjective. Then
\begin{enumerate}[(i)]
\item
$\rank \on{j}_L(Q,P) = \dim \on{im} QP^\top = \dim \on{im} P^\top = m$,
the second equality following from the injectivity of $Q$.
\item
%
%Since $Q$ is injective, it restricts to a bijection $Q:\ker P^\top Q \to \ker P^\top\cap \on{im}Q$. Let $q=\dim \ker P^\top Q = \dim (\ker P^\top\cap \on{im} Q)$. Then $q \le \dim\ker P^T = n-m$, and so
%\[
%\rank \on{j}_R(Q,P) = \dim \on{im} P^\top Q = m - \dim \ker P^T Q = m - q \ge 2m-n.
%\]
$\rank \on{j}_R(Q,P) =\dim \on{im} P^\top Q =m-\dim \on{ker} P^\top Q\ge m- \dim \on{ker} P^\top
=2m-n$, where we use  that $\dim\ker P^\top = n-m$ by the rank-nullity theorem.
\end{enumerate}
\end{proof}

In fact, the momentum maps $\on{j}_L, \on{j}_R$ are surjective onto $S_L, S_R$:

\begin{proposition}\phantomsection\label{prop:surjectivity}
\begin{enumerate}[(i)] 
\item $\gl(n,\RR)_{\on{j}_L} = S_L$.
\item $\gl(m,\RR)_{\on{j}_R} = S_R$.
\end{enumerate}
\end{proposition}

\begin{proof}
\begin{enumerate}[(i)]
\item

Let $\zeta \in S_L$. We wish to show $\zeta$ is in the image of $\on{j}_L$. By left $\GL(n,\RR)$-equivariance of $\on{j}_L$, we may assume $\zeta$ is in (real) Jordan canonical form
\begin{equation}\label{eqn:zetacan}
\zeta = \begin{bmatrix} J_{c_1}(\lambda_1) \\ & \ddots \\ && J_{c_p}(\lambda_p) \\ &&& J_{d_1}(0) \\ &&&&\ddots \\ &&&&& J_{d_q}(0) \\ &&&&&& 0_{(n-m-q)\times (n-m-q)}\end{bmatrix} \in \gl(n,\RR),
\end{equation}
where $J_{c_i}(\lambda_i)\in\on{M}_{c_i\times c_i}(\mathbb{R})$ denotes the (real) $i$th Jordan block corresponding to \emph{non-zero} generalised eigenvalue $\lambda_i$, and $J_{d_j}(0)\in \on{M}_{d_j\times d_j}(\mathbb{R})$ is the $j$th \emph{non-trivial} (i.e., with $d_j\ge 2$) Jordan block corresponding to \emph{zero} generalised eigenvalues. The dimension of the zero block follows from the condition
\begin{equation}\label{eqn:zetarank}
m = \rank \zeta = \sum_{i=1}^p c_i + \sum_{j=1}^q (d_j-1),
\end{equation}
which implies $n - \sum_{i=1}^p c_i - \sum_{j=1}^q d_j = n-m-q.$

To construct a suitable $(Q,P)\in \MnmRm\times\MnmRm$ mapping to $\zeta$ under $\on{j}_L$: for $d\ge 2$, let $I_{d-1}\in \on{M}_{(d-1)\times(d-1)}(\mathbb{R})$ denote the identity matrix, $\check{I}_{d-1} = \begin{bmatrix} I_{d-1} \\ 0_{1\times (d-1)} \end{bmatrix}\in \on{M}_{d\times(d-1)}(\mathbb{R})$, and $\hat{I}_{d-1} = \begin{bmatrix} 0_{1\times (d-1)} \\ I_{d-1} \end{bmatrix}\in\on{M}_{d\times(d-1)}(\mathbb{R})$. It is straightforward to check that
\begin{equation} \label{eqn:nilpotent1}
\check{I}_{d-1} \hat{I}_{d-1}^\top = J_d(0).
\end{equation}
Take
\begin{equation}\label{eqn:QP}
Q = \begin{bmatrix} J_{c_1}(\lambda_1) \\ & \ddots \\ && J_{c_p}(\lambda_p)  \\ &&& \check{I}_{d_1-1} \\ &&&& \ddots \\ &&&&& \check{I}_{d_q-1} \\ && 0_{(n-m-q)\times m}\end{bmatrix}, \qquad
P = \begin{bmatrix} I_{c_1} \\ & \ddots \\ && I_{c_p} \\ &&& \hat{I}_{d_1-1} \\ &&&& \ddots \\ &&&&& \hat{I}_{d_q-1} \\ &&& 0_{(n-m-q)\times m} \end{bmatrix}.
\end{equation}
By \eqref{eqn:zetarank}, $(Q,P)\in\MnmRm\times \MnmRm$. Using \eqref{eqn:QP}, \eqref{eqn:nilpotent1}, and \eqref{eqn:zetacan},  it may be checked that $\on{j}_L(Q,P) = QP^\top = \zeta$.

\item
Let $\xi \in S_R$. Again, by right $\GL(m,\RR)$-equivariance of $\on{j}_R$ we may assume that $\xi$ is in (real) Jordan canonical form
\begin{equation}\label{eqn:xican}
\xi = \begin{bmatrix} J_{c_1}(\lambda_1) \\ & \ddots \\ && J_{c_p}(\lambda_p) \\ &&& J_{d_1-1}(0) \\ &&&&\ddots \\ &&&&& J_{d_q-1}(0) \end{bmatrix} \in \gl(m,\RR),
\end{equation}
where $c_i, d_j$, and $\lambda_i$ obey the same conventions as in part (i) (in particular, $d_j\ge 2$), but where  now we let $J_1(0)$ denote the $1\times1$ zero matrix instead of explicitly writing a zero block.
Defining $\check{I}_{d-1}, \hat{I}_{d-1}$ as before, we have
\begin{equation}\label{eqn:nilpotent2}
\hat{I}_{d-1}^\top \check{I}_{d-1} = J_{d-1}(0).
\end{equation}
From the condition $2m-n \le \rank \xi = m-q$, implying $n-m-q\ge 0$, 
we see that the matrices \eqref{eqn:QP} are well-defined, with $(Q,P)\in\MnmRm\times \MnmRm$. Using  \eqref{eqn:QP}, \eqref{eqn:nilpotent2}, and \eqref{eqn:xican}, it may be checked that $\on{j}_R(Q,P) = P^\top Q = \xi$.

\end{enumerate}
\end{proof}

%We point out that the integers $q$ in the proofs of Lemma \ref{lemma:immom} and Proposition \ref{prop:surjectivity} are the same. 

Examining the proof of Proposition \ref{prop:surjectivity} gives an explicit characterisation of the adjoint orbit correspondence for our dual pair:

\begin{corollary}
Given 
\begin{itemize}
\item integers $p,q$ with $0\le p\le m$ and $0\le q\le \on{min}\lbrace m, n-m\rbrace$;
\item complex numbers $\lambda_1,\ldots, \lambda_p$,  with $\on{Im} \lambda_i \ge 0$;
\item integers $c_1, \ldots, c_p, d_1,\ldots, d_q$ satisfying $c_i\ge 1$ and $c_i$ even if $\on{Im} \lambda_i > 0$, $d_j\ge 2$, and $\sum_{i=1}^p c_i +\sum_{j=1}^q (d_j-1) = m$.
\end{itemize}
Then there is a one-to-one correspondence between the adjoint orbits through elements \eqref{eqn:zetacan} in $\gl(n,\RR)_{\on{j}_L}$ and \eqref{eqn:xican} in $\gl(m,\RR)_{\on{j}_R}$.
\end{corollary}

\subsection{Relations between the \texorpdfstring{$(\GL(n,\RR),\GL(m,\RR))$}{(GL(n,R),GL(m,R))} and \texorpdfstring{$(\Sptn,\Om)$}{(Sp(2n,R),O(m))} momentum maps}

We can realize $\gl(n,\RR)$ as a Lie subalgebra of $\sp(2n,\RR)$
with the map 
$$\ell:\zeta\in\gl(n,\RR)\mapsto\begin{bmatrix} \zeta& 0 \\ 0& -\zeta^\top\end{bmatrix}
\in\sp(2n,\RR).$$
Denoting by $i$ the inclusion of $\om$ into $\gl(m,\RR)$, we obtain:

\begin{proposition}The diagram
\[
\xymatrix{
\sp(2n,\RR)^*\ar[d]_{\ell^*} &\MtnmR \ar[l]_{\J_{\Sptn}}\ar[r]^{\J_{\Om}} \ar[d]_{=}&
{\mathfrak{o}(m)}^*\\
\gl(n,\RR)^*& \,\,T^*\MnmR\,\,\ar[l]_{{\J}_{\GL(n,\RR)}} \ar[r]^{{\J}_{\GL(m,\RR)}} &
\gl(m,\RR)^*\ar[u]_{i^*}\\
}
\]
commutes, where here momentum maps are labelled by their corresponding groups.
\end{proposition}

\begin{proof}
Let $E=\begin{bmatrix} Q \\ P\end{bmatrix}
%\left(\begin{array}{ccc} Q \\P   \end{array}\right)
\in \MtnmR=T^*\MnmR$ be arbitrary.
For all $\zeta\in\gl(n,\RR)$ we have
\begin{align*}
\langle\ell^*(\J_{\Sptn}(E)),\zeta\rangle
&=\frac12\Omega(\ell(\zeta)E,E)=\frac12\Tr(E^\top\ell(\zeta)^\top \mathbb{J}E)=
\frac12\Tr(Q^\top \zeta^\top P+ P^\top \zeta Q)\\
&=\Tr(QP^\top \zeta)=\langle \J_{\GL(n,\RR)}(E),\zeta\rangle.
\end{align*}
The identity $i^*\circ\on{J}_{\GL(n,\RR)} = \on{J}_{\Om}$ is proved similarly.
%Also, for all $\xi\in {\om}$ we get
%\begin{align*}
%\langle \J_{\Om}(E),\xi\rangle
%&=\frac12\Omega(E\xi,E)=\frac12\Tr(\xi^\top E^\top \mathbb{J}E)=-\frac12\Tr(E^\top \mathbb{J}E\xi)
%=\frac12\Tr((P^\top Q-Q^\top P)\xi)\\&=\frac12\Tr(P^\top Q(\xi-\xi^\top))=\Tr(P^\top Q\xi)=\langle i^* \J_{\GL(m,\RR)}(E),\xi\rangle.
%\end{align*}
\end{proof}

%%%%%%%%%%%%%%%%%%%%
%\newpage
%%%%%%%%%%%%%%%%%%

\subsection*{Acknowledgements} 
Both authors wish to thank Tillman Wurzbacher and Reyer Sjamaar for useful discussions, and the referees for careful reading and informative comments. We particularly wish to thank the third referee for flagging an error in the original proof of Proposition \ref{prop:dualpairrank}, and for suggestions to clarify the proofs of Section \ref{sec:adjcorr}. The first author was partially supported by Leverhulme Trust Research Project Grant 2014-112. The second author was supported by the grant PN-III-P4-ID-PCE-2016-0778 of the Romanian Ministry of Research and Innovation CNCS-UEFISCDI, within PNCDI III.

%%%%%%%%%%%%%%%%%%%

{\footnotesize

\bibliographystyle{new}

\end{document}